\documentclass[10pt]{amsart}
\usepackage{amsmath}
\usepackage{amsthm,amsfonts,amssymb,mathrsfs}
\usepackage{epic,eepic}
\usepackage{yfonts}
\usepackage{paralist,enumerate}

\usepackage{amsmath,amsfonts,amsthm,amssymb,mathrsfs,enumerate,comment, amscd}
\usepackage{cite}

\usepackage[all]{xy}
\usepackage{epsfig,color}
\usepackage{pdftricks}
\usepackage{setspace}

\usepackage{epsfig,color}

\theoremstyle{plain}
\newtheorem{theorem}{Theorem}

\newtheorem*{lemma*}{Lemma}
\newtheorem{lemma}[theorem]{Lemma}

\newtheorem*{theorem*}{Theorem}

\newtheorem*{proposition*}{Proposition}

\newtheorem{corollary}[theorem]{Corollary}
\newtheorem*{corollary*}{Corollary}

\theoremstyle{definition}

\newtheorem{remark}[theorem]{Remark}
\newtheorem*{remark*}{Remark}

\newtheorem*{definition*}{Definition}

\newtheorem*{example*}{Example}

\newtheorem*{question*}{Question}

\newcommand{\cO}{\mathcal O}

\numberwithin{equation}{section}

\def\RR{{\mathbb R}}

\def\c1{\operatorname{c_1}}
\def\c2{\operatorname{c_2}}

\def\ZZ{{\mathbb Z}}

\def\QQ{{\mathbb Q}}
\def\PP{{\mathbb P}}

\def\O{{\mathcal O}}

\def\+{\oplus}                   
\def\*{\otimes}                  
\def\rat{\dashrightarrow}       

\def\Cr{\operatorname{Cr}}

\def\NEb{\overline{\operatorname{NE}}}

\def\Pic{\operatorname{Pic}}

\usepackage{mathtools}

\renewcommand\setminus{-}

\title{Curves disjoint from a nef divisor}

\author[J. Lesieutre]{John Lesieutre}
\address{Institute for Advanced Study, Einstein Drive, Princeton, New Jersey 08540 USA}
\email{johnl@math.ias.edu}

\author[J.C. Ottem]{John Christian Ottem}

\address{DPMMS, University of Cambridge, Wilberforce Road, Cambridge CB3 0WB, UK}
\email{J.C.Ottem@dpmms.cam.ac.uk}

\begin{document}

\begin{abstract}
On a  projective surface it is well-known that the set of curves orthogonal to a nef line bundle is either finite or uncountable. We show that this dichotomy fails in higher dimension by constructing a nef line bundle on a threefold which is trivial on countably infinitely many curves. This answers a question of Totaro. As a pleasant corollary, we exhibit a quasi-projective variety with only a countably infinite set of complete, positive-dimensional subvarieties.

\end{abstract}

\maketitle

\thispagestyle{empty}

\section{Introduction}
If $L$ is a nef line bundle on a smooth complex projective surface, then the set of curves $C$ such that $L\cdot C=0$ is either finite or uncountable (when some such $C$ moves in a positive-dimensional family). This follows essentially from the Hodge index theorem. In \cite{totaro}, Totaro asked whether this remains true in higher dimensions:

\begin{question*}
\label{totaroq}
Is there a nef line bundle $L$ on a normal complex projective variety $X$ such that the set
of curves $C$ with $L\cdot C = 0$ is countably infinite?
\end{question*}

In this note we construct examples of such \(L\) in dimensions greater than two. Perhaps the surprising thing  is not that such examples exist, but that they turn out to be so accessible: in fact, our example is the blow-up of $\PP^3$  at eight very general points and $L$ is the anticanonical divisor. 
Our main result is the following:

\begin{theorem}\label{nefexample}
There exists a smooth projective rational threefold $X$ with nef anticanonical divisor so that the set of curves $C$ with $-K_X\cdot C=0$ is countably infinite.
\end{theorem}

In particular, since $-K_X$ is effective in the example, the complement of the zero set of a global section gives an example of the following:
\begin{corollary}\label{weirdquasiproj}
There exists a quasi-projective variety with only a countably infinite set of complete, positive-dimensional subvarieties.
\end{corollary}

Furthermore, we show that the question has an affirmative answer even
if the line bundle is required to be big and nef, which is impossible
in dimension less than four (cf. Remark \ref{3fold}). 

\begin{corollary}\label{nullexample}
  There exists a smooth projective fourfold \(Y\) and a big and nef
  line bundle \(M\) on \(Y\) so that the set of curves \(C\) with \(M
  \cdot C=0\) is countably infinite. 
\end{corollary}

Note that when the line bundle $L$ is semiample, these sorts of pathologies do
not occur. In that case, some multiple of $L$ defines a morphism to
projective space $X\to \PP^N$ which contracts exactly the curves
orthogonal to $L$, so this locus is Zariski closed. In particular, if
\(X\) is a Mori dream space, every nef line bundle is semiample and so
is zero on an either finite or uncountable set of curves. This
includes all examples in which \(X\) is the blow-up of \(\PP^3\) at $r\le
7$ points.

\section{The rational threefold}

Let $p_1,\ldots,p_8$ be eight very general points in $\PP^3$. The linear system of quadric surfaces containing the points is a one-dimensional pencil. Let $Q_0,Q_1$ be two distinct smooth quadric surfaces in this pencil. The base-locus of the pencil, $B=Q_0\cap Q_1$, is a smooth genus \(1\) curve, which is a bidegree $(2,2)$ divisor on the quadrics.

Define $\pi:X\to \PP^3$ to be the blow-up of $\PP^3$ at the points
$p_1,\ldots,p_8$ and let $B'$ be the strict transform of $B$ on
\(X\). The pencil of quadrics determines a rational map
$f:X\dashrightarrow \PP^1$, defined outside the curve $B'$. The fibers
of this map are blow-ups of quadric surfaces in the 8 points
$p_1,\ldots,p_8$.

Write $H=\pi^\ast \O_{\PP^3}(1)$ and let $E_1,\ldots,E_8$ be the exceptional divisors of $\pi$.  Similarly, let \(h = H^2\) be the class of the strict transform of a general line in $\PP^3$, and let  $e_1,\ldots,e_8$ denote classes of lines in  $E_1,\ldots, E_8$ (which are projective planes). We have $N^1(X)=\ZZ H \oplus \ZZ E_1\oplus \cdots \oplus \ZZ E_8$ and $N_1(X)=\ZZ h\oplus \ZZ e_1\oplus \cdots \oplus \ZZ e_8$.

Let $L=-K_X$ be the anticanonical divisor of $X$: this is the nef divisor we are looking for. In terms of this basis, the canonical divisor is given by $-4H+2E_1+\ldots+2E_8$. Note that $L$ is nef, since its base-locus is exactly the curve $B'$ and $L \cdot B' = -K_X\cdot B'=(-K_X)^3=0$. Note also that the fibers of $f$ correspond to divisors in the linear system $\lvert-\frac12K_X\rvert$.

There are many curves $C$ such that $L\cdot C=0$. For example, let \(C\) be the strict transform of the line $l$ through the points $p_1$ and $p_2$. The class of \(C\) is \(h-e_1-e_2\), and we have $$L\cdot C=(4H-2E_1-\ldots-2E_8)\cdot (h-e_1-e_2)=0.$$ Let $Q$ be the quadric surface in the pencil containing $l$ as one of its rulings. Then the strict transform $S$ of $Q$ is the blow-up of $Q$ in $p_1,\ldots,p_8$, and the curve $C$ is a $(-2)$-curve on $S$ with class $\pi^*\O(0,1)-E_1-E_2$ in $\Pic(S)$.

The same thing happens if we take $C$ to be the strict transform of a twisted cubic curve in $\PP^3$ through six of the points $p_1,\ldots,p_6$; the class of $C$ is $3h-e_1-\ldots-e_6$, and \(L \cdot C = 0\).  There is a unique quadric surface $Q$ in the pencil containing the twisted cubic, and its strict transform $S$ contains $C$ as a $(-2)$-curve.

In fact, we will show below that there are countably infinitely many curves $C$ on $X$ such that $L\cdot C=0$: these will be constructed as the strict transforms of the lines in $\PP^3$ through a pair of points and under sequences of Cremona transformations on $\PP^3$ based at quadruples of points. These are all rigid rational curves with normal bundle isomorphic to $\O(-1)\oplus \O(-1)$. Moreover, the set of these curves is Zariski dense in $X$.

\begin{lemma}
\label{minus2curves}
Let $S$ be a smooth rational surface with $-K_S$ nef and $K_S^2=0$. If $C$ is  an irreducible curve such that $K_S\cdot C=0$, then either $C\in \lvert-mK_S\rvert$ for some $m\ge 1$ or $C$ is a smooth rational curve of self-intersection $-2$.
\end{lemma}

\begin{proof}
The Hodge index theorem implies that $C^2\le 0$. If $C^2<0$, then $C\cdot K_S=0$ implies $C^2=-2$ and $p_a(C)=0$ by the adjunction formula, and hence $C\simeq \PP^1$. So suppose $C^2=0$. For any $C'\in K_S^\perp$, either $C'\cdot C=0$, or we have $(tC+C')^2>0$ for some $t>0$. In the latter case we again obtain a contradiction using the Hodge index theorem. Thus we have that $C^\perp=K_S^\perp$ and hence $C=-m K_S$ for some $m\ge 1$, since $-K_S$ is not divisible in $\Pic(S)$.
\end{proof}

So far we have not used the fact that the points $p_1,\ldots,p_8$ are {\em very} general on $B$. For us, the crucial fact is that in this case there are no relations in $\Pic(B)$ between line bundles in the very general quadric in the pencil and the points $p_1,\ldots,p_8$. 

One way of seeing this is the following: Fix a smooth quadric surface $Q$ in the pencil and let $M$ be a line bundle on $Q$. For each set of integers $a_1,\ldots,a_8$ such that $M'=M(-a_1p_1-\ldots-a_8p_8)|_B$ has degree 0 on $B$, there is a Zariski closed subset of points $(p_1,\ldots,p_8)\in B^8$ such that the $M'$ is effective  on $B$ (that is, $M'|_B=\O_B$). Now we can assume that the points $p_1,\ldots,p_8$ are chosen outside the countable union of all these closed subsets running through all the choices $M,a_1,\ldots,a_8$. Then no non-trivial line bundle on $Q$ restricts to the trivial bundle on $B$. We have essentially also shown the following

\begin{lemma}\label{Pic-injective}
Let $B\subset \PP^3$ be a smooth quartic curve. Then for a very general $Q$ in the pencil $|I_B(2)|$, the restriction map $\Pic(Q)\to \Pic(B)$ is injective.
\end{lemma}

We are now ready to prove the main result of this section:

\begin{lemma}
\label{atmostcountable}
  The set of curves \(C \subset X\) for which \(L \cdot C = 0\) is at
  most countably infinite.
\end{lemma}
\begin{proof}
  Suppose that \(C\) is a curve with \(L \cdot C = 0\).  If \(C\) is
  not contained in the base locus of \(\lvert-\frac12K_X\rvert\), then
  \(C\) meets some fiber \(S \in \lvert-\frac{1}{2}K_X\rvert\) of the
  map \(f:X \dashrightarrow \PP^1\), at a point not contained in
  \(B^\prime\).  If \(C\) is not contained in the fiber \(S\), then
  \(S \cdot C\) is positive and so too is \(L \cdot C\).
  Consequently, if \(C\) is a curve with \(L \cdot C = 0\), then
  either \(C\) is the unique curve \(B^\prime\) in the base locus, or
  \(C\) lies in some fiber \(S\) of \(f\). In the following, we will
  assume that $C\neq B'$.

Assume first that $S$ is smooth. $S$ is the strict transform of a
smooth quadric, i.e., the blow-up of $\PP^1\times \PP^1$ along  8
points. By Lemma \ref{minus2curves}, $C$ is either linearly equivalent to a multiple of
$-K_{S}=B'$, or is a $(-2)$-curve on $S$. However, the normal bundle
of $B'$ in $S$ is $$\O_{S}(B')|_{B'}=\O(-K_{S})|_{B'}=\O(S)|_{B'}=
\O_{\PP^3}(2)(-p_1-\ldots-p_8)|_{B},$$ which is non-torsion for
very general $p_1,\ldots,p_8$, and so no multiple of $B'$ moves in $S$. It follows that the only curve on $S$ with class proportional to $-K_S$ is $B'$ itself. We conclude that $C \subset S$ is a $(-2)$-curve.   Since \(-K_S\) is nef, the number of \((-2)\)-curves on \(S\) is finite.

The family \(f\) also has four singular fibers \(S_s\), each
isomorphic to an 8-point blow-up of a quadric cone. By genericity of
the points, we may assume that the curve $B^\prime$ does not pass
through the singular point of any of these fibers. Let \(\sigma :
\tilde{S}_s \to S_s\) be the blow-up at the singular point, so that
\(\tilde{S}_s\) is isomorphic to the blow-up of \(\PP(\O\+\O(2))\) at
\(8\) points.  Then \(\sigma^\ast(B^\prime)\) is anticanonical and
\(\tilde{S}_s\) is a smooth rational surface with \(-K_{\tilde{S}_s}\)
nef and \(K_{\tilde{S}_s}^2 = 0\). It follows that
\(-K_{\tilde{S}_s}\) has no movable multiple, by the same argument
used to prove this in the smooth fibers. So the strict transform of
\(C\) on \(\tilde{S}_s\) must be a \((-2)\)-curve, and as before there
are only finitely many such curves since \(-K_{\tilde{S}_s}\) is nef.

Since the set of $(-2)$-curves in any fiber is finite, we are reduced
to showing that there are only countably many smooth fibers of $f$
containing $(-2)$-curves.

Note that $C$ corresponds to a divisor through the points $p_1,\ldots,p_8$ on some quadric surface $Q$. Restricting the section defining $C$ in $Q$ to $B$ gives a relation in $\Pic^0(B)\simeq B$ between the points $p_1,\ldots,p_8$ and line bundles coming from $Q$. However, by Lemma \ref{Pic-injective} there are only countably many fibers where this happens.
\end{proof}

\def\p{\mathbf p}
\section{Cremona actions}

Additional curves with \(L \cdot C = 0\) will be constructed using
repeated applications of the standard Cremona transformation on
\(\PP^3\), yielding `elementary $(-1)$-curves', considered by Laface
and Ugaglia~\cite{lu2}.
The standard Cremona transformation $\Cr:\PP^3\dashrightarrow \PP^3$ is given by 
$$
\Cr(x_0,x_1,x_2,x_3)=(x_0^{-1},x_1^{-1},x_2^{-1},x_3^{-1})
$$

Let \(\pi : X \to \PP^3\) be the blow-up of \(\PP^3\) at the four
standard coordinate points.  The rational map $\Cr \circ \pi : X \rat
\PP^3$ can be factored as follows.
\[
\xymatrix{
 & Y \ar[dl]_p \ar[dr]^{p^\prime} & \\
X \ar@{-->}[rr]^{\overline{\Cr}} \ar[d]_\pi 
&&  X^\prime \ar[d]^{\pi^\prime} \\
\PP^3 \ar@{-->}[rr]^{\Cr} && \PP^3
}
\]
Here $p$ is the blow-up of $X$ along the transforms of the six lines through pairs of the four coordinate points.  The exceptional divisors of $p$ are isomorphic to $\PP^1\times \PP^1$ and $p'$ is the contraction of the `other ruling' of each $\PP^1\times \PP^1$. The induced map $\overline{\Cr}$ is a flop of these curves. $\pi'$ then blows down the strict transforms of the four planes through three of the four points, realizing $X^\prime$ as the blow-up of $\PP^3$ at four points as well.

The Cremona transformation has the following properties: (i)
$\overline{\Cr}$ is an isomorphism in codimension 1, (ii) It
preserves the canonical class (i.e.,
$\overline{\Cr}^*(K_{X^\prime})=K_{X}$) and (iii) it induces
isomorphisms $M : N^1(X) \to N^1(X^\prime)$ and $\check{M} : N_1(X)
\to N_1(X^\prime)$, given in the standard bases by the matrices
$\left(\begin{smallmatrix}M &0 \\ 0 & I_4 \end{smallmatrix}\right)$
and $\left(\begin{smallmatrix}\tilde M &0 \\ 0 &
    I_4 \end{smallmatrix}\right)$ where
\[
M = \left( \begin{array}{rrrrr}
3 & 1 & 1 & 1 & 1 \\
-2 & 0 & -1 & -1 & -1 \\
-2 & -1 & 0 & -1 & -1 \\
-2 & -1 & -1 & 0 & -1 \\
-2 & -1 & -1 & -1 & 0 \\
\end{array} \right), \quad
\tilde{M} = \left( \begin{array}{rrrrr}
3 & 2 & 2 & 2 & 2 \\
-1 & 0 & -1 & -1 & -1 \\
-1 & -1 & 0 & -1 & -1 \\
-1 & -1 & -1 & 0 & -1 \\
-1 & -1 & -1 & -1 & 0 \\
\end{array} \right).
\]

If $\p=(p_1,\ldots,p_8)$ is an 8-tuple of distinct points in $\PP^3$ with the first four not coplanar, we denote by $\Cr_\p : \PP^3 \rat \PP^3$ the transformation $A^{-1}\circ \Cr \circ A$ where $A$ is the linear transformation taking $p_1,\ldots,p_4$ to the standard coordinate points. (If the points are in general position, $A$ is uniquely determined if we additionally impose that it also fixes the point $(1,1,1,1)$).  Write \(\mathbf q\) for the new \(8\)-tuple \((p_1,\ldots,p_4,\Cr_{\mathbf p}(p_5),\ldots,\Cr_{\mathbf p}(p_8))\).  

Let \(X_{\mathbf p}\) denote the blow-up of \(\PP^3\) at the eight
points of \(\mathbf p\), and \(X_{\mathbf q}\) denote the blow-up of
  \(\PP^3\) at the eight points of \(\mathbf q\).  The discussion
  above shows that the map \(\Cr_{\mathbf p} : \PP^3 \rat \PP^3\)
  induces a birational map \(\overline{\Cr}_{\mathbf p} : X_{\mathbf
    p} \rat X_{\mathbf q}\), which flops the six lines between two of
  the four points \(p_1,\ldots,p_4\).

  The crucial observation is that a very general configuration of 8
  points in $\PP^3$ has infinite orbit under the group generated by
  Cremona transformations. This fact was essentially known to Coble
  \cite{coble}; see \cite{dolgachev} for a more modern account.

\section{Proof of Theorem 1}
We are now in position to complete the proof of Theorem 1. Again, we let $X$ denote the blow-up of $\PP^3$ in a very general configuration $\mathbf p=(p_1,\ldots, p_8)$ of eight points. We have
already seen that the set of curves $C$ such that $L\cdot C=0$
correspond to $(-2)$-curves on the fibers of $f:X\dashrightarrow
\PP^1$ and that this set is at most countably infinite. It remains
only to show that this set is in fact infinite.

\begin{lemma}
\label{actuallycountable}
There is an infinite set of curves \(C \subset X\) with \(L \cdot C = 0\).
\end{lemma}
\begin{proof}
  Starting from the very general configuration \(\mathbf p_0 = \mathbf
  p\), construct a sequence of configurations \(\mathbf p_0,\mathbf
  p_1,\ldots,\mathbf p_n,\ldots\) so that \(\mathbf p_{i-1}\) is
  obtained from \(\mathbf p_i\) by making a Cremona transformation
  centered at the first four points of \(\mathbf p_i\) and then
  permuting the \(8\)-tuple to move the first entry to the end of the
  list. The ``very general'' assumption on $\mathbf p_0$ guarantees
  that no four points ever become coplanar, and so the requisite
  Cremona transformations are well-defined.

This gives rise to a sequence of rational maps
\[
\xymatrix{
\cdots \ar@{-->}[r]^-{\Cr_{\mathbf p_{n+1}}} &
X_{\mathbf p_n} \ar@{-->}[r]^-{\Cr_{\mathbf p_{n}}} &
X_{\mathbf p_{n-1}} \ar@{-->}[r]^-{\Cr_{\mathbf p_{n-1}}} &
\cdots \ar@{-->}[r]^-{\Cr_{\mathbf p_{1}}} &
X_{\mathbf p_0} = X
}
\]

If \(C\) is a curve on \(X_{\mathbf p_n}\) such that the strict
transform of \(C\) on \(X_{\mathbf p_i}\) is disjoint from the
indeterminacy locus of \(X_{\mathbf p_i} \rat X_{\mathbf p_{i-1}}\)
for all \(1 \leq i \leq n\), then the strict transform of \(C\) on
\(X_{\mathbf p}\) has numerical class \(\tilde{M}_\sigma^n([C])\),
where

\setlength{\arraycolsep}{2pt}
\[
\tilde{M}_\sigma =  \left( \begin{array}{c|c}
\tilde{M} & 0 \\\hline
0 & I_4
\end{array}
\right)
\left( \begin{array}{c|c}
1 & 0 \\\hline
0 & \Pi_\sigma
\end{array}
\right) = 
 \begin{tiny}\left(
\begin{array}{rrrrrrrrr}
3 & 2 & 2 & 2 & 2 & 0 & 0 & 0 & 0 \\ 
-1 & -1 & 0 & -1 & -1 & 0 & 0 & 0 & 0 \\ 
-1 & -1 & -1 & 0 & -1 & 0 & 0 & 0 & 0 \\ 
-1 & -1 & -1 & -1 & 0 & 0 & 0 & 0 & 0 \\ 
0 & 0 & 0 & 0 & 0 & 1 & 0 & 0 & 0 \\ 
0 & 0 & 0 & 0 & 0 & 0 & 1 & 0 & 0 \\ 
0 & 0 & 0 & 0 & 0 & 0 & 0 & 1 & \phantom{-}0 \\ 
0 & \phantom{-}0 & \phantom{-}0 & \phantom{-}0 & \phantom{-}0 & \phantom{-}0 & \phantom{-}0 & \phantom{-}0 & 1 \\ 
-1 & 0 & -1 & -1 & -1 & 0 & 0 & 0 & 0
\end{array}\right)
\end{tiny}
\]
with \(\Pi_\sigma\) the matrix encoding the permutation of the points.

If we take \(\ell_n\) to be a line through \(p_7\) and \(p_8\) on
\(X_{\mathbf p_n}\), its strict transform on \(X_{\mathbf p}\) is a
curve \(C_n\) of class \(\tilde{M}_\sigma^n(h-e_7-e_8)\); that the
strict transforms of \(\ell_n\) are disjoint from the indeterminacy
loci is checked in \cite{lu2}. It is easy to verify
that the matrix \(\tilde{M}_\sigma\) has a \(3 \times 3\) Jordan block
associated to the eigenvalue \(1\), and a direct calculation then
shows that the degrees of the classes
\([C_n]=\tilde{M}_\sigma^n(h-e_7-e_8)\) grow without bound as \(n\) is
increased, so the curves \(C_n\) are distinct.

However, for every value of \(n\) we have \(-K_{X_{\mathbf p}} \cdot
C_n = 0\): the curve \(\ell_n \subset X_{\mathbf p_n}\) is a rational
curve with normal bundle \(\cO(-1) \oplus \cO(-1)\),
and the same is true of its strict transform \(C_n \subset X_{\mathbf
  p}\) because \(\ell_n\) does not meet the indeterminacy locus of the
map \(X_{\mathbf p_n} \rat X_{\mathbf p}\).
It follows that these give an infinite set of curves with \(L \cdot C
= 0\).  
\end{proof}

\begin{remark}
It is well-known that the classes in \(K_X^\perp\) on a point-set blow-up of projective space form a root system (in our case, it is the T-shaped Dynkin diagram $T_{4,4,2}$), and the Cremona transformations induce elements in the corresponding Weyl group. Moreover, the curves on which $K_X$ is zero are exactly the orbit of the class of a line in \(N_1(X)\) under this Weyl group. (See \cite{dolgachev} or \cite{Mukai} for more precise statements).  The composition of the Cremona transformation and a permutation of the points used above corresponds to the action of a Coxeter element in this group.
\end{remark}

A more detailed account can be found in~\cite{lesieutre1}: the
curves here are (up to permutation of the indices) the curves
``\(C_n\)'' constructed in Lemma 5.2. Although~\cite{lesieutre1} deals
with the blow-up of \(\PP^3\) at \(9\) points, the same argument works
with only \(8\); the only difference is that the matrix
\(\tilde{M}_\sigma\) considered here has a \(3 \times 3\) Jordan block
associated to the eigenvalue \(1\), rather than an eigenvalue greater
than \(1\).

Together, Lemmas \ref{atmostcountable} and \ref{actuallycountable}
complete the proof of Theorem 1.  The corollaries stated in the
introduction follow immediately.

\begin{proof}[Proof of Corollary~\ref{weirdquasiproj}]
  Fix a very general smooth representative \(S\) of
  \(\lvert-\frac12K_X\rvert\), and let \(U = X \setminus S \subset
  X\). It is clear that every complete curve $C$ in $U$ must satisfy
  $S \cdot C = -\frac{1}{2}K_X\cdot C=0$, and we have already shown
  that the set of curves with this property is countably
  infinite. Moreover, none of these curves \(C\) except \(B^\prime\)
  is contained in \(S\), but they all satisfy \(S \cdot C = 0\).
  Consequently all such \(C\) are contained in \(U\).
\end{proof}

\begin{proof}[Proof of Corollary~\ref{nullexample}]
  Let \(H\) be a very ample divisor on \(X\) and consider the variety
  \(Y = \PP(\cO_X \oplus \cO_X(H))\).  The fourfold \(Y\) admits two
  obvious maps: first, there is a \(\PP^1\)-bundle \(p : Y \to X\);
  second, there is a contraction \(q : Y \to CX\) of the section \(E
  \subset Y\) determined by the quotient \(\cO_X \oplus \cO_X(H) \to
  \cO_X\), yielding the projective cone \(CX\) over \(X\).

  Fix an ample divisor \(G\) on \(CX\), and take \(M = p^\ast(L) +
  q^\ast(G)\).  The pullback \(p^\ast(L)\) is certainly nef, and since
  \(q\) is birational and \(G\) is ample, \(q^\ast(G)\) is big and
  nef. The line bundle \(M\), being the sum of a nef line bundle and a
  big and nef one, is itself big and nef.

  Suppose now that \(C\) is a curve with \(M \cdot C = 0\).  It
  must be that \(q^\ast(G) \cdot C = 0\), so \(C\) is contracted by
  \(q\), and lies in the exceptional section \(E \subset Y\).  Under
  the identification \(E \cong X\), the restriction \(M\vert_E = p^\ast(L)
  \vert_E\) is simply the line bundle \(L\), and so the set of curves
  \(C \subset E\) with \(M \cdot C = 0\) is countable.
\end{proof}

The same construction using $\PP(\cO_X \oplus \cO_X(H)^{\oplus r})$ with
  $r\ge 1$ gives examples as in Corollary 3 for any dimension greater than three.
  
\begin{remark}\label{3fold}
  If \(L\) is a big and nef line bundle on a threefold \(X\), then the
  set of curves with \(L \cdot C = 0\) is either finite or
  uncountable.  Indeed, \(L\) is \(\QQ\)-linearly equivalent to a sum
  \(A+E\), with \(A\) ample and \(E\) effective.  Any curve with
  \((A+E) \cdot C = 0\) must be contained in the support of \(E\).
  For any component \(E_i \subset E\), the divisor \(L \vert_{E_i}\)
  is nef and hence zero on either finitely many or uncountably many
  curves; this follows from the two-dimensional statement, applied on
  a resolution of \(E_i\).

\end{remark}

\section{Remarks}

\subsection{Blow-ups of $\PP^1\times \PP^1\times \PP^1$}
We obtain a similar example by considering a 6-point blow-up of $\PP^1\times \PP^1\times \PP^1$. Here the canonical divisor on the blow-up $X$ is given by $2D$ where $D=\pi^*\O(1,1,1)-E_1-\ldots-E_6$. Again there is a 1-dimensional family of $(1,1,1)$-divisors passing though the 6 points. Each $(1,1,1)$-divisor corresponds to a Del Pezzo surface of degree 6 (in fact each projection to $\PP^1\times \PP^1$ is the blow-up of $\PP^1\times \PP^1$ in 2 points). It follows that $X$ is fibered into blow-ups of $\PP^1\times \PP^1$ in 8 points, as before.

Again there are many curves on $X$ so that $-K_X\cdot C=0$. For example, when an exceptional divisor of a Del Pezzo surface passes through one of the points, the strict transform is a $(-2)$-curve on $X$ which satisfies $K_X\cdot C=0$. Here an infinite sequence of such curves can be obtained by applying the Cremona transformations of the form
$$
\phi:(x_0,x_1) \times (y_0,y_1)\times (z_0,z_1) \mapsto (x_1,x_0)\times (y_0/{x_0},y_1/{x_1})\times (z_0/{x_0},z_1/{x_1}) 
$$This transformation, and its permutations, generates an infinite representation in $GL(N^1(X)_\RR)$, as shown by Mukai in \cite{Mukai}, and so arguing as before we obtain infinitely many curves on $X$ such that $-K_X\cdot C=0$. We note that this threefold is not isomorphic to the previous example.

\subsection{A question}

 The example here shows that it is possible for a linear subspace of
\(N_1(X)\) to contain precisely a countable number of irreducible
curves: \(-K_X^\perp \subset N_1(X)\) is such a subspace.
Since \(-K_X\) is nef, \(-K_X^\perp \cap \NEb(X)\) is in fact an
extremal face of the cone of curves \(\NEb(X)\) containing a countable number of
irreducible curves. Related is the following:
 
\begin{question*}Let $X$ be a smooth projective variety and let
  $\alpha\in N_1(X)$ be a numerical cycle class. Can it happen that
  the set of irreducible curves on $X$ with class proportional to
  $\alpha$ is countably infinite?
\end{question*}

Again, this can not happen on a surface. Indeed, a divisor $D$ on a surface either has a movable multiple (in which the number is uncountable) or $h^0(mD)=1$ for all $m\ge 1$. In the latter case, arguing as in \cite{ottem} or \cite{totaro} shows that the number of irreducible divisors is less than the Picard number of the surface.


\begin{thebibliography}{10}

\bibitem{coble}
Arthur~B. Coble, \emph{Algebraic geometry and theta functions}, American
  Mathematical Society Colloquium Publications, vol.~10, American Mathematical
  Society, Providence, R.I., 1982, Reprint of the 1929 edition.

\bibitem{dolgachev}
Igor Dolgachev and David Ortland, \emph{Point sets in projective spaces and
  theta functions}, Ast\'erisque (1988), no.~165, 210 pp. (1989).



\bibitem{lu2}
Antonio Laface and Luca Ugaglia, \emph{Elementary {$(-1)$}-curves of {$\PP ^3$}}, Comm. Algebra
  \textbf{35} (2007), no.~1, 313--324.
%
%

\bibitem{lesieutre1}
John Lesieutre, \emph{The diminished base locus is not always closed}, to appear in Compositio Mathematica. arXiv:1212.3738

%
%


\bibitem{Mukai}
Shigeru Mukai, \emph{Counterexample to {H}ilbert's fourteenth problem for the
  3-dimensional additive group}, RIMS Kyoto preprint \textbf{1343} (2001).


\bibitem{Nakamaye}Michael Nakamaye, Stable base loci of linear series. Mathematische Annalen 318 (2000), 837--847.

\bibitem{ottem}
John Christian Ottem, \emph{On subvarieties with ample normal bundle}, to appear in Journal of the European Mathematical Society, arXiv:1309.2263.
  
%

\bibitem{totaro}
Burt Totaro. Moving codimension-one subvarieties over finite fields, Amer. J. Math. 131 (2009), 1815--1833.

\end{thebibliography}
\end{document}